\documentclass{amsart}
\usepackage{amsmath,amssymb,amsthm}
\newtheorem{thr}{Theorem}
\newtheorem{lem}[thr]{Lemma}
\newtheorem{cor}[thr]{Corollary}

\theoremstyle{definition}
\newtheorem{defn}[thr]{Definition}

\newtheorem{prob}[thr]{Problem}

\theoremstyle{remark}

\numberwithin{equation}{section}

\begin{document}

\title[On the complexity of local graph coloring]{On the complexity of graph coloring \\ with additional local conditions} 

\author{Yaroslav Shitov}
\address{National Research University Higher School of Economics, 20 Myasnitskaya Ulitsa, Moscow 101000, Russia}
\email{yaroslav-shitov@yandex.ru}

%    \subjclass is required.

\begin{abstract}
Let $G=(V,E)$ be a finite simple graph. Recall that a \textit{proper coloring} of $G$ is a mapping $\varphi:V\to\{1,\ldots,k\}$ such that every color class induces an independent set. Such a $\varphi$ is called a \textit{semi-matching coloring} if the union of any two consecutive color classes induces a matching. We show that the semi-matching coloring problem is NP-complete for any fixed $k\geqslant 3$, and we get the same result for another version of this problem in which any triangle of $G$ is required to have vertices whose colors differ at least by three. %This is a solution of a recent problem by Li, Shao, Zhu, Xu.
\end{abstract}

\maketitle

In this paper, we establish the algorithmic complexity of the two graph coloring problems mentioned in the abstract. They have attracted some attention in recent publications, including the paper~\cite{Haj} by Hajiabolhassan on a semi-matching generalization of Kneser's conjecture. The other notion, known in the literature as a \textit{local $k$-coloring}, has been studied in several papers both from the point of view of Kneser's conjecture and as a concept deserving an independent interest (see~\cite{CSSZ, LZSX, OP} and references therein). As a reformulation of what is said in the abstract, one can define a local $k$-coloring as a mapping $\varphi:V\to\{1,\ldots,k\}$ such that for each set $S$ of either $2$ or $3$ vertices, there exist $u, v \in S$ such that the number $|\varphi(u)-\varphi(v)|$ is greater than or equal to the number of edges passing between the vertices in $S$. Li, Shao, Zhu, Xu proved in~\cite{LZSX} that the local $k$-coloring problem is NP-complete for $k=4$ and for any fixed odd $k\geqslant 5$; they posed a problem to determine the complexity in the remaining cases of $k=3$ and even $k\geqslant 6$. We solve this problem, and actually we prove the NP-completeness of both the local and semi-matching coloring problems for arbitrary fixed $k$ in a unified manner.

\section{Preliminaries}

Throughout the rest of our paper, we assume that $k\geqslant 3$ is a fixed integer. For ease of reference, we formulate the problems we are going to study below (these problems would become trivial if $k$ was less than $3$).

\begin{prob}\label{prob11} Given: A finite simple graph $G$.

\noindent Question 1: Is there a local $k$-coloring of $G$?

\noindent Question 2: Is there a semi-matching $k$-coloring of $G$?
\end{prob}

The main result of this paper is as follows.

\begin{thr}\label{thrmain}
Questions~1 and~2 in Problem~\ref{prob11} are NP-complete.
\end{thr}

It is easy to see that both questions in Problem~\ref{prob11} belong to NP, and we are going to prove their NP-hardness by constructing polynomial reductions from the problem known as NAE $3$-SAT. Recall that a \textit{literal} (corresponding to a set $X$ of Boolean variables) is either an element of $X$ or a negation of it.

\begin{prob}\label{prob13} (NAE $3$-SAT.)

\noindent Given: A set $X$ of variables and a set $L$ of triples of literals.

\noindent Question: Does there exist an assignment $X\to\{\mbox{True},\mbox{False}\}$ for which every triple in $L$ has at least one false literal and at least one true literal?
\end{prob}

\section{Earlier results that we use}

First of all, we recall that NAE $3$-SAT is an NP-complete problem (see~\cite{Scha}), so the existence of polynomial reductions from Problem~\ref{prob13} to the questions in Problem~\ref{prob11} would indeed imply Theorem~\ref{thrmain}. Further, let us recall that the \textit{chromatic number} of a graph $G$ is the smallest integer $\chi$ such that $G$ admits a proper $\chi$-coloring; the local and semi-matching chromaric numbers are defined analogously for the corresponding types of colorings. The following appears as Theorem~2.1 in~\cite{LSZX}.

\begin{thr}\label{thrcon}
For any integer $\tau\geqslant1$, there exists a graph whose chromatic number, semi-matching chromatic number, and local chromatic number are all equal to $\tau$.
\end{thr}

%We will need the following easy corollary of Theorem~\ref{thrcon}.

\begin{cor}\label{corcon}
For any integer $\tau\geqslant1$, there exists a graph $\Gamma_\tau$ such that the chromatic number, semi-matching chromatic number, and local chromatic number of $\Gamma_\tau$ and any graph obtained from $\Gamma_\tau$ by removing at most two vertices are all equal to $\tau$.
\end{cor}

\begin{proof}
Take a disjoint union of three copies of the graph as in Theorem~\ref{thrcon}.
\end{proof}

For any positive integers $n,r$ satisfying $n\geqslant 2r$, we denote by $C(n,r)$ the graph obtained from a complete graph on $n$ vertices by removing $r$ non-adjacent edges. %According to the abstract of the paper~\cite{CSSZ}, its authors determine the local chromatic numbers of complete multipartite graphs. Let me give a specific version of this result without a proof (I did not manage to find a copy of~\cite{CSSZ} online, but we do not need a general formulation of their result but rather a very special case which can be easily proved by induction on $n$).

\begin{thr}\label{thrc} \rm{(See~\cite{CSSZ}.)}
The local chromatic number of $C(n,r)$ is $\lfloor 1.5n-0.5\rfloor - r$.
\end{thr}

%We are now ready to proceed with the proof of Theorem~\ref{thrmain}.

\section{The equality gadget}

In this section, we construct polynomial reductions to local and semi-matching $k$-colorings from more general problems that we are ready to present. (We recall that $k\geqslant 3$ is an integer parameter fixed in advance.)

\begin{prob}\label{prob14} $ $

\noindent Given: A simple graph $G=(V,E)$; a sequence $U$ of subsets $U_1,\ldots,U_t\subset V$.

\noindent Question: Is there a local $k$-coloring of $G$ in which, for every $i\in\{1,\ldots,t\}$, all the vertices of $U_i$ get the same color, and this color is either $1$ or $k$?
\end{prob}

Let us define a `semi-matching version' of Problem~\ref{prob14} by replacing the word `local' with the word `semi-matching' in its formulation. We define the graph $\mathcal{G}(G,U)$ by adjoining to $G$ the $t$ copies of the graph $\Gamma_{k-2}$ as in Corollary~\ref{corcon} and drawing the edges between every vertex in $U_i$ and every vertex in the corresponding (that is, $i$th) copy of $\Gamma_{k-2}$. Clearly, $\mathcal{G}$ can be computed in polynomial time as a function of $G$ and $U$.

\begin{lem}\label{lemred}
$\mathcal{G}$ is a reduction (i) from Problem~\ref{prob14} to Question~1 in Problem~\ref{prob11}, (ii) from the semi-matching version of Problem~\ref{prob14} to Question~2 in Problem~\ref{prob11}.
\end{lem}

\begin{proof}
If we have a coloring of $G$ as in Problem~\ref{prob14}, then we can extend it to $\mathcal{G}$ as follows. Namely, the vertices in $U_i$ got the same color $1$ (or $k$), and we produce the local $(k-2)$-coloring of the $i$th copy of $\Gamma_{k-2}$ using indexes $3,4,\ldots,k$ (or $1,2,\ldots,k-2$, respectively). It is straightforward to check that the resulting assignment is a local $k$-coloring of $\mathcal{G}$ (or a semi-matching coloring if we consider the semi-matching version of Problem~\ref{prob14}).

Conversely, assume that we have a local (or semi-matching) $k$-coloring of $\mathcal{G}$. If one of the vertices in one of the $U_i$'s has color $c$, then no vertex in the $i$th copy of $\Gamma_{k-2}$ can have color $c$, at most one vertex in that copy can have index $c-1$, and at most one vertex in the copy can have index $c+1$. Therefore, if there are two vertices in $U_i$ of different colors, or there is a vertex of color different from $1$ and $k$, we get a proper $(k-3)$ coloring of a graph obtained from $\Gamma_{k-2}$ by removing two vertices. This contradicts Corollary~\ref{corcon} and shows that the vertices in every $U_i$ get the same color, and this color is either $1$ or $k$.
\end{proof}

\section{The NAE gadget}

This section is devoted to auxiliary results that will allow us to encode the NAE $3$-SAT problem as an instance of Problem~\ref{prob14}. We need to define a sequence of graphs which we denote by $L(k)$; we set $L(2\tau+2)=C(2\tau,\tau)$ and $L(2\tau+3)=C(2\tau,\tau-1)$ for any integer $\tau\geqslant 1$, and we define $L(3)$ to be the graph with one vertex.

\begin{defn}\label{defnLK}
We construct the graph $\mathcal{L}_k(u_1,u_2,v)$ as follows. We take two disjoint copies of $L(k)$ and three new vertices $u_1,u_2,v$, and draw new edges from $u_1$ to every vertex of the first copy, from $u_2$ to every vertex of the second copy, and from $v$ to every vertex in both copies.
\end{defn}

\begin{lem}\label{lemLK1}
$\mathcal{L}_k(u_1,u_2,v)$ has no local $k$-coloring in which $u_1,u_2,v$ have colors $1,1,k$ or $k,k,1$, respectively.
\end{lem}

\begin{proof}
Assume that such a coloring exists, and the vertices $u_1,u_2,v$ are assigned colors $1,1,k$ (the $k,k,1$ case is analogous). If there were two more vertices $a,b$ colored either $k-1$ or $k$ each, then $(a,v,b)$ would be a path colored with two consecutive indexes. This contradicts the definition of a local coloring, so one of the copies of $L(k)$ has colors $1,\ldots,k-2$ only. However, the union of this copy and the corresponding vertex $u_i$ induces the subgraph which does not admit a local $(k-2)$-coloring as we can check by Theorem~\ref{thrc}.
\end{proof}

\begin{lem}\label{lemLK2}
Any mapping $\varphi:\{u_1,u_2,v\}\to\{1,k\}$ can be extended to a local $k$-coloring of  $\mathcal{L}_k(u_1,u_2,v)$ unless $\varphi(u_1), \varphi(u_2), k+1-\varphi(v)$ are all equal.
\end{lem}

\begin{proof}
There are only two possible cases up to symmetry.

\textit{Case 1.} If $\varphi(u_1)=\varphi(u_2)=\varphi(v)=k$, then we can use $1,\ldots,k-2$ to get a local $k$-coloring of both copies of $L(k)$, which is possible by Theorem~\ref{thrc}.

\textit{Case 2.} Now let $\varphi(u_1)=\varphi(v)=k$ and $\varphi(u_2)=1$. Again, we construct a local coloring of the first copy of $L(k)$ using $1,\ldots,k-2$ only. Our strategy for the second copy depends on the parity of $k$; if $k=2\tau+2$, then we color the vertices of $L(k)=C(2\tau,\tau)$ so that the endpoints of removed edges get indexes $(2,3)$, $(4,5)$, $\ldots$, $(k-2,k-1)$, thus completing a local $k$-coloring of $\mathcal{L}_k(u_1,u_2,v)$.

If $k=2\tau+3$, then we color the vertices of $L(k)=C(2\tau,\tau-1)$ so that the vertices not adjacent to removed edges get indexes $2$ and $k-1$, and the removed edges are $(4,4)$, $(6,6)$, $\ldots$, $(2\tau,2\tau)$. Actually, this assignment is a local $k$-coloring of $\mathcal{L}_k(u_1,u_2,v)$ not only when $\tau\geqslant 1$ but also when $\tau=0$, but anyway the case of $k=3$ can be easily treated separately.
\end{proof}

Now let us define the graph $\mathcal{SM}_k(u_1,u_2,v)$ similarly to $\mathcal{L}_k(u_1,u_2,v)$ but with $L(k)$ replaced by the complete graph on $k-2$ vertices. In this case, Lemmas~\ref{lemLK1} and~\ref{lemLK2} will stay true if we replace $\mathcal{L}$ by $\mathcal{SM}$ and `local' by `semi-matching' in their formulations. We do not provide the proofs because they can be easily reconstructed from the arguments above.

\section{Completing the proof}

\begin{thr}\label{lem55}
Problem~\ref{prob14} is NP-hard.
\end{thr}

\begin{proof}
We construct the instance $(G,U)$ of Problem~\ref{prob14} depending on an instance $(X,L)$ of Problem~\ref{prob13} as follows. 

\textit{Step 1.} For every variable $x\in X$, we add vertices corresponding to $x$ and $\overline{x}$ and connect them by an edge.

\textit{Step 2.} For every tuple $(a,b,c)\in L$, we add a copy of the graph $\mathcal{L}_k(u_a,u_b,v_{\overline{c}})$. 

\textit{Step 3.} We define $2|X|$ subsets $U_\chi$ corresponding to the variables in $X$ and their negations. Namely, we define $U_\chi$ as a union of the vertex corresponding to $\chi$ as in Step~1 and all the vertices $u_\chi$ and $v_{\chi}$ as in Step~2.

Clearly, the mapping $(X,L)\to (G,U)$ can be computed in polynomial time. Let us check that it is a reduction; assume that $\varphi$ is a coloring of $G$ as in Problem~\ref{prob14}. The edges added in Step~1 guarantee that the classes $U_\chi$ and $U_{\overline{\chi}}$ should get opposite colors in $\{1,k\}$, so we can think of the color choice of every $U_x$ as a truth assignment of a variable $x\in X$. Since the copies of $\mathcal{L}_k(u_a,u_b,v_{\overline{c}})$ are disjoint, $\varphi$ is a local coloring if and only if its restriction to every such copy is local, which is possible if and only if $a,b,c$ are not all equal, according to Lemmas~\ref{lemLK1} and~\ref{lemLK2}.
\end{proof}

Replacing every appearance of $\mathcal{L}$ by $\mathcal{SM}$ and `local' by `semi-matching' in the proof above, we get the NP-hardness of the semi-matching version of Problem~\ref{prob14}. Applying Lemma~\ref{lemred}, we complete the proof of Theorem~\ref{thrmain}.

%\section{Acknowledgment}

\bigskip

This research was financially supported by the Russian Science Foundation (project No. 17-71-10229).

\end{document}